\documentclass[11pt,a4paper]{amsart}

\usepackage{geometry}
\usepackage[english]{babel}
\usepackage{amsmath,amssymb,amsthm,amscd}
\usepackage{url}
\usepackage{color,hyperref}
\usepackage{pst-plot} 

\theoremstyle{plain}
\newtheorem{theorem}{Theorem}[section]
\newtheorem{proposition}[theorem]{Proposition}
\newtheorem{lemma}[theorem]{Lemma}
\newtheorem{corollary}[theorem]{Corollary}

\theoremstyle{definition}

\newtheorem{remark}[theorem]{Remark}
\newtheorem{example}[theorem]{Example}

\numberwithin{equation}{section}

\newcommand{\N}{\mathbb{N}}
\newcommand{\Z}{\mathbb{Z}}
\newcommand{\Q}{\mathbb{Q}}
\newcommand{\R}{\mathbb{R}}
\newcommand{\set}[2]{\{#1\,|\ #2\}}
\newcommand{\sub}{\subseteq}

\begin{document}
\title[Congruence-simple matrix semirings]{Congruence-simple matrix semirings}

\author[V.~Kala]{V\'{i}t\v{e}zslav~Kala}
\address{Department of Algebra, Faculty of Mathematics and Physics, Charles University, Sokolovsk\'{a} 83, 186 75 Prague 8, Czech Republic}
\email{vitezslav.kala@matfyz.cuni.cz}

\author[T.~Kepka]{Tom\'{a}\v{s}~Kepka}
\address{Department of Algebra, Faculty of Mathematics and Physics, Charles University, Sokolovsk\'{a} 83, 186 75 Prague 8, Czech Republic}
\email{kepka@karlin.mff.cuni.cz}

\author[M.~Korbel\'a\v{r}]{Miroslav~Korbel\'a\v{r}}
\address{Department of Mathematics, Faculty of Electrical Engineering, Czech Technical University in Prague, Technick\'{a} 2, 166 27 Prague 6, Czech Republic}
\email{korbemir@fel.cvut.cz}

\thanks{The first author was supported by Czech Science Foundation (GA\v{C}R), grant 21-00420M, and Charles University Research Centre program UNCE/SCI/022. The third author acknowledges the support by the bilateral Austrian Science Fund (FWF) project I 4579-N and Czech Science Foundation (GA\v{C}R) project 20-09869L ``The many facets of orthomodularity''}

\date{\today}


\begin{abstract}
It is well known that the full matrix ring over a skew-field is a simple ring. We generalize this theorem to the case of semirings. We characterize the case when the matrix semiring $\mathbf{M}_n(S)$, of all $n\times n$ matrices over a semiring $S$, is congruence-simple, provided that either $S$ has a multiplicatively absorbing element or  $S$ is commutative and  additively cancellative.
\end{abstract}

\maketitle
\vspace{4ex}

Classifying algebraic structures  that are \textit{simple} in the sense that they do not possess any non-trivial congruences is one of the general goals of algebra. Perhaps the most notorious example is the classification of finite simple groups, but also the case of rings is well-studied. The Wedderburn–-Artin theorem gives a characterization in the semisimple case in terms of products of matrix rings. 
In particular, it is well-known that for a simple ring $R$  (whether unitary or not, but with non-trivial multiplication) the full matrix ring $\mathbf{M}_n(R)$, of square  $n\times n$ matrices over $R$, is again a simple ring. 
So, for instance, the ring of all matrices over the rational numbers $\Q$ is simple. 

The broad goal of this paper is to study analogous questions for \textit{semirings}, i.e., sets $S$ equipped with two associative operations $+$ and $\cdot$ such that $+$ is commutative and the distributive property holds.
Semirings have been rising in prominence in the last few decades, especially thanks to their uses in fields such as cryptography  and tropical mathematics \cite{durcheva,ims,litvinov,mmr}.

In the case of \textit{commutative} multiplication, the classification of congruence-simple semirings is mostly complete \cite{simple_comm}, with the exception of subsemirings of positive real numbers $\mathbb R^+$. Let us also briefly note that the \textit{ideal-simple} commutative case turned out to be quite tricky \cite{JKK, KalKor}.

Non-commutative congruence-simple semirings $S$ are much less understood, despite some partial results \cite{simple, monico}.
In particular, a basic classification \cite[Theorem 2.1]{simple} shows that  $S$ is additively cancellative, idempotent, or zeropotent (or the addition  $+$ is trivial), and there are some further results in these cases. In particular, much more is known when $S$ is finite \cite{monico}.  However, the general classification program is likely to be very hard, and so it is necessary to first find a wide range of  examples and possible constructions. 

Matrix semirings again present such a natural class of candidates for simplicity.
While other features of matrices over semirings are already known in more details (e.g., \cite{zumbragel,shitov}),
we are missing a characterization for when the matrix semiring is simple, in contrast with the ring case.
 Even for subsemirings of $\Q^+$, the semiring of positive rational numbers,  the corresponding matrix semirings are not so well understood.
 
\bigskip

The aim of this paper is to  fill this gap. 
Our main result  is a characterization of when the semiring $\mathbf{M}_n(S)$ is congruence-simple  under the assumption that 
\begin{itemize}
 \item[-]  $S$ has a multiplicatively absorbing element (Theorem \ref{6.4}), or
  \item[-]   $S$ is additively cancellative and downwards directed (Theorem \ref{theorem_cancellative_2}), or
 \item[-]  $S$ is additively cancellative and commutative  (Theorem \ref{theorem_cancellative}).
\end{itemize}

Our theorems generalize \cite[Proposition 5.5]{KNZ-2} that concerned the case of semirings with zero and unity. 
The restriction to semirings that have both of these special elements usually allows one to obtain results similar as in the ring case, for instance, the generalization of the Wedderburn--Artin theorem  \cite{KNZ-2,KNZ}. Also the question of when a given semiring $T$ is isomorphic to a full matrix semiring  $\mathbf{M}_n(S)$ for some semiring $S$ was answered in several ways \cite{bogdanov,golan2,robson} for semirings with  zero and  unity.
Moreover, the well-known correspondence between  the semiring $\mathbf{M}_n(S)$ and all $S$-endomorphisms of a certain $S$-semimodule holds precisely in this case (see, e.g., Proposition \ref{unity_1}).

On the other hand, this is certainly not the only interesting case and semirings without a zero element or without a unity also deserve attention (as evidenced, e.g., by the subsemirings of $\R^+$ or by certain lattices). The contribution of our paper is mainly for those general cases of semirings. We also present examples of  semirings (Example \ref{ex_1}) allowing the corresponding matrix semirings to be congruence-simple but without both a zero and a unity.

\section{Preliminaries}

A \emph{semiring} $S$ for our consideration will be a non-empty set equipped with two associative binary operations, usually denoted as addition $+$ and multiplication $\cdot$ such that the addition is commutative and the multiplication distributes over the addition from both sides.

Let $S=S(+,\cdot)$ be a semiring. For  subsets $A,B$ of $S$ we set $A+B=\set{a+b}{a\in A, b\in B}$ and $AB=\set{ab}{a\in A, b\in B}$. Further, we will use the \emph{standard quasiordering} $\leq_S$ on $S$ that is defined by $a\leq_S b$ if $a=b$ or if there is $c\in S$ such that $a+c=b$. Clearly, $\leq_S$ is transitive and stable with respect to the operations of the semiring $S$.

A non-empty subset $I$ of $S$ is a \emph{bi-ideal} of $S$ if $a+s,a s, s a\in I$ for every $a\in I, s\in S$. 
A \textit{congruence} on a semiring $S$ is an equivalence relation that is preserved under the addition and multiplication.

The semiring $S$ is called
\begin{itemize}
 \item \emph{congruence-simple} if $S$ has just two congruences;
 \item \emph{bi-ideal-simple} if $|S|\geq 2$ and for every bi-ideal $I$ of $S$ either $|I|=1$ or $I=S$;
 \item \emph{bi-ideal-free} if $|S|\geq 2$ and the only bi-ideal of $S$ is $S$ itself;
 \item \emph{commutative} if the semigroup $S(\cdot)$ is commutative;
 \item \emph{additively idempotent} if $a+a=a$ for every $a\in S$;
 \item \emph{additively cancellative} if $a+c\neq b+c$ for all $a,b,c\in S$ such that $a\neq b$;
 \item \emph{additively archimedean} if for all $a,b\in S$ there is  $m\in\N$ such that $a\leq_S mb$;
  \item \emph{additively (uniquely) $m$-divisible} (for $m\in\N$) if for every $a\in S$ there is (a unique) $b\in S$ such that $a=mb$.
\end{itemize}

An element $w\in S$ is called
\begin{itemize}
\item \emph{multiplicatively} (\emph{additively}, resp.) \emph{absorbing} if $aw=w=wa$ ($a+w=w$, resp.) for every $a\in S$; 
\item \emph{a unity} (\emph{additively neutral}, resp.)  if $aw=a=wa$ ($a+w=a$, resp.) for every $a\in S$ (a unity will be denoted by $1_S$);
\item \emph{bi-absorbing} if $w$ is both multiplicatively and additively absorbing (such an element will be denoted by $o_S$);
\item a \emph{zero} if $w$ is multiplicatively absorbing and additively neutral (such an element will be denoted by $0_S$).
\end{itemize}

Notice that the element $w\in S$ is bi-absorbing if and only if $\{w\}$ is a bi-ideal (and therefore $S$ has at most one trivial bi-ideal).

Also, for every bi-ideal $I$ of $S$, the relation $(I\times I)\cup id_S$ is a congruence on $S$. Hence if $S$ is congruence-simple, then $S$ is also bi-ideal-simple.

For some further information on semirings see, e.g., \cite{G,weinert}.

\bigskip

Let $S$ be a semiring and $n\geq 1$. By $\mathbf{M}_n(S)$ we denote the full matrix semiring of (square) $n\times n$-matrices
 with entries from $S$ with the usual matrix operations: addition $(a_{ij})+(b_{ij})=(a_{ij}+b_{ij})$ and multiplication $(a_{ij})(b_{ij})=(\sum_{k}a_{ik}b_{kj})$.

 For every $a\in S$, let $\overline{a}=(a_{ij})\in \mathbf{M}_n(S)$ be the matrix such that $a_{ij}=a$ for all $i,j$. Clearly, $\overline{a}+\overline{b}=\overline{a+b}$ and $\overline{a}\overline{b}=n\overline{ab}$. Thus $\overline{S}=\set{\overline{a}}{a\in S}$ is a subsemiring of $\mathbf{M}_n(S)$.

\bigskip

 First, we investigate the case when a matrix semiring possesses certain significant elements (a unity, a zero, a bi-absorbing or multiplicatively absorbing element).

 \begin{proposition}\label{2.5}
 Let $S$ be a semiring and $n\geq 1$. Then for $T=\mathbf{M}_n(S)$ the following holds:
 \begin{enumerate}
 \item[(i)] $0_T\in T$ if and only if $0_S\in S$. If one of the conditions is fulfilled, then $0_T=\overline{0_{S}}$.

 \item[(ii)]  $o_T\in T$ if and only if $o_S\in S$. If one of the conditions is fulfilled, then $o_T=\overline{o_S}$.
 \end{enumerate}
 \end{proposition}
\begin{proof}
 (i) Let $0_T=(\omega_{ij})$ be a zero element in $T$. Then for every $A=(a_{ij})\in T$ it holds that $\omega_{ij}+a_{ij}=a_{ij}$ for all $i,j$. Hence $\omega_{ij}=\omega\in S$ is an additively neutral element in $S$.

 Now choose $b\in S$. Since $0_T$ is multiplicatively absorbing in $T$, we have $0_T\cdot B=0_T$ for every $B=(b_{ij})\in T$ such that $b_{11}=b$ and $b_{ij}=\omega$ for all $(i,j)\neq(1,1)$. Hence $\omega=\omega\cdot b+\omega\cdot \omega+\cdots+\omega\cdot\omega=\omega(b+\cdots+\omega)=\omega\cdot b$.  By a dual argument, we obtain that $\omega=b\cdot\omega$ for every $b\in S$.  Hence $\omega=0_S$ is a zero element in $S$. The rest is easy.

 (ii) Let $o_T=(\omega_{ij})\in T$. Then for every $A=(a_{ij})\in T$ it holds that $\omega_{ij}+a_{ij}=\omega_{ij}$ for all $i,j$. Hence $\omega_{ij}=\omega\in S$ is additively absorbing in $S$.

 Further, $o_T\cdot \overline{b}=o_T$ for every $b\in S$. Hence $\omega=\omega\cdot b+\cdots+\omega\cdot b=\omega(nb)=(n\omega)\cdot b=\omega\cdot b$. Thus $\omega$ is left multiplicatively absorbing. By a dual argument, we obtain that $\omega$ is right multiplicatively absorbing.  Hence $\omega=o_S$ is a bi-absorbing element in $S$. The rest is easy.
\end{proof}

\begin{proposition}\label{2.4}
  Let $n\geq 2$ and $S$ be a semiring.  Then  the following are equivalent:
 \begin{enumerate}
  \item[(i)] $\mathbf{M}_n(S)$ has a zero and a unity.
  \item[(ii)] $\mathbf{M}_n(S)$ has a unity.
  \item[(iii)] $S$ has a zero and a unity.
 \end{enumerate}
 If one of the condition is fulfilled, then the unity $1_T=(e_{ij})$ of $T=\mathbf{M}_n(S)$ is of the form $e_{ii}=1_S$ and $e_{ij}=0_S$ for $i\neq j$.
\end{proposition}
\begin{proof}
 (i) implies (ii) trivially and (iii) implies (i) by Proposition \ref{2.5}(i).

 (ii)$\Rightarrow$(iii): Let $1_T=(e_{ij})$ be a unity. For $1\leq i\leq n$ set $R_{i}=\set{\sum_{j,j\neq i}e_{ij}a_j}{a_j\in S}\sub S$. Clearly, $R_i$ is a right ideal of the semiring $S$, (i.e., $R_i\cdot S\sub R_i$ and  $R_i+R_i\sub R_i$). As $1_T$ is the unity of $T$, we have $e_{ii}a+u=a$ for every $i$ and all $a\in S$ and $u\in R_i$. Consequently, $a=e_{ii}a+u+v=a+v$ for all $a\in S$ and $u,v\in R_i$. Hence there is an additively neutral element in $w\in S$ and $R_i=\{w\}$ for all $i$. Thus, $\{w^2\}\sub wS=R_iS\sub R_i=\{w\}$ and it follows that $w^2=w$ and $wS=\{w\}$.

 Similarly, $L_j=\set{\sum_{i,i\neq j} a_ie_{ij}}{a_i\in S}$  is a left ideal of $S$ (i.e., $S\cdot L_j\sub L_j$ and  $L_j+L_j\sub L_j$) and, by a symmetrical argument, we obtain that $L_j=\{w\}$ for all $j$ and $Sw=\{w\}$.
This  means that $w=0_S$ is a zero element in $S$.

Finally, for every $i$ and all $a\in S$ and $v\in R_i=\{0_S\}$  we have $a=e_{ii}a+v$. Therefore $a=e_{ii}a$.  Similarly, it holds that $a=ae_{ii}$.  We conclude that $e_{ii}=1_S$ is a unity in $S$ for every $i$.  Of course, $e_{ij}=0_S$ for $i\neq j$.
\end{proof}

\begin{proposition}\label{unity_1}
Let $n\geq 2$ and $S$ be a semiring.  Then  the following are equivalent:
 \begin{enumerate}
 \item[(i)] There are a semiring $R$ and a left $R$-semimodule $_RM$ such that the semiring $\mathrm{End}(_RM)$ (of all $R$-endomorphisms  of $_RM$) is isomorphic to $\mathbf{M}_n(S)$.
 \item[(ii)] $S$ has a zero and a unity.
 \end{enumerate}
\end{proposition}
\begin{proof}
(i)$\Rightarrow$(ii): As the semiring $\mathrm{End}(_RM)$ contains the identical map $id_M$, the semiring $\mathbf{M}_n(S)$ contains a unity. Hence, by Proposition \ref{2.4}, $S$ has a zero and a unity.

(ii)$\Rightarrow$(i):
Let $S$ have a zero and a unity. Set $R=S$ and $M=S\times \cdots \times S$ ($n$-times). Then the left $S$-semimodule $_SM$ is free with a free basis $e_1,\dots,e_n$ (of a standard form). Hence for every $S$-endomorphism $\varphi$ of $_SM$ there is a unique matrix $A\in \mathbf{M}_n(S)$ such that $\varphi(m)=m\cdot A$ for every $m\in\, _SM$. It follows that $\mathrm{End}(_SM)\cong \mathbf{M}_n(S)$. 
\end{proof}

\begin{proposition}\label{2.6}
 Let $n\geq 1$. Then the set $S_{n}=S+\cdots+ S$ ($n$-times) is a bi-ideal of $S$  and the following conditions are equivalent:
 \begin{enumerate}
  \item[(i)] $\mathbf{M}_n(S)$ has a multiplicatively absorbing element.
  \item[(ii)] $S_{n}$ has a multiplicatively absorbing element (if $w$ is such an element, then $\overline{w}$ is the multiplicatively absorbing element of $\mathbf{M}_n(S)$).
 \end{enumerate}
 If moreover $S+S=S$, then $S_n=S$ and the conditions above are equivalent to the condition 
 \begin{enumerate}
 \item[(iii)]  $S$ has a multiplicatively absorbing element.
 \end{enumerate}
\end{proposition}
\begin{proof}
 (i)$\Rightarrow$(ii): Let  $W=(w_{ij})$ be a multiplicatively absorbing element in $\mathbf{M}_n(S)$. Then $w_{i1}a_1+\cdots+w_{in}a_n=w_{ij}$ for all $a_1,\dots,a_n\in S$ and all $i,j$. Hence for  $k\neq j$ it holds that $w_{ik}=w_{i1}a_1+\cdots+w_{in}a_n=w_{ij}$.   Consequently, $w_{i1}=\cdots=w_{in}=w_i$ and $w_i S_{n}=\{w_i\}$. Dually, $a_1w_1+a_2w_2+\cdots+a_nw_n=w_i$ for all $i$ and $a_1,\dots,a_n\in S$. Thus $w_1=w_2=\cdots=w$ and $w S_{n}=\{w\}=S_{n}w$.

The rest of the implications is easy.
\end{proof}

\section{Congruence-simple matrix semirings}

In this section we provide a characterization of congruence-simple matrix semirings that do contain a multiplicatively absorbing element. This result will be a generalization of \cite[Proposition 5.5]{KNZ-2}, where the semirings with both zero and a unity element are assumed.

Let us start by recalling basic properties of congruence-simple semiring from \cite{simple}.

\begin{lemma}\cite[Lemma 2.2]{simple}\label{6.0}
 Let $S$ be a congruence-simple semiring such that $|SS|\geq 2$. Then at least one of the following conditions holds:
 \begin{enumerate}
 \item[(i)] $|S|=2$, $S$ is both multiplicatively and additively idempotent and $S$ has no multiplicatively absorbing element.
 \item[(ii)] For every $a,b\in S$ such that $a\neq b$ there are $c,d\in S$ such that $ca\neq cb$ and $ad\neq bd$.
 \end{enumerate}
\end{lemma}

\begin{proposition}\cite[Proposition 3.4]{simple}\label{6.0.1}
 Let $S$ be a semiring. If the semiring $\mathbf{M}_n(S)$ is congruence-simple for some $n\geq 2$, then:
  \begin{enumerate}
  \item[(i)] $S$ is congruence-simple.
  \item[(ii)] $S$ contains no bi-absorbing element.
  \item[(iii)] $|SS|\geq 2$.
  \item[(iv)] $S$ is either additively idempotent or additively cancellative (and therefore $|S+S|\geq 2$).
 \end{enumerate}
\end{proposition}

\begin{theorem}\label{unity}
 \cite[Theorem 5.4]{simple} Let $S$ be an additively cancellative congruence-simple semiring.
 \begin{enumerate}
  \item[(i)]  If $S$ has an additively neutral elemeny then $S$ is a ring.
  \item[(ii)]  If $S$ has a unity, then $S$ is either a ring or $S$ is commutative.
 \end{enumerate}
\end{theorem}

 Further, the following lemma is easy to verify.

\begin{lemma}\label{6.2}
Let $S$ be a semiring with a zero element $0_S$ and $n\geq 2$.
For $i_0,j_0\in\{1,\dots,n\}$ and $a\in S$ we denote by $E_{i_0j_0}(a)=(e_{ij})\in \mathbf{M}_n(S)$ a matrix such that $e_{i_0j_0}=a$ and $e_{i'j'}=0_S$ for all $(i',j')\neq(i_0,j_0)$.

The following holds for all $c,d\in S$, $i,j,k,\ell\in\{1,\dots,n\}$ and $A=(a_{ef})\in \mathbf{M}_n(S)$:
 \begin{enumerate}
 \item[(i)] $E_{ij}(c)+E_{ij}(d)=E_{ij}(c+d)$.
  \item[(ii)] $E_{ii}(c)\cdot E_{ij}(d)=E_{ij}(c\cdot d)$ and $E_{ij}(d)\cdot E_{jj}(c)=E_{ij}(d\cdot c)$.
  \item[(iii)]
  $E_{ik}(c)\cdot A\cdot E_{\ell j}(d)=E_{ij}(c\cdot a_{k \ell}\cdot d)$.
 \end{enumerate}
\end{lemma}

Now we present our first result. 

\begin{theorem}\label{6.4}
 Let $S$ be a semiring and let for some $k\geq 1$ the semiring $\mathbf{M}_k(S)$ have a multiplicatively absorbing element. Then the following conditions are equivalent:
  \begin{enumerate}
  \item[(i)] $\mathbf{M}_n(S)$ is congruence-simple for every $n\geq 1$.
  \item[(ii)] $\mathbf{M}_n(S)$ is congruence-simple for at least one $n\geq 2$.
  \item[(iii)] $S$ is congruence-simple, $S$ contains no bi-absorbing element and $|SS|\geq 2$.
 \end{enumerate}
 If one of these conditions is fulfilled, then $S$ has a zero element and $S$ is either a ring or additively idempotent.
\end{theorem}
\begin{proof}
 (i)$\Rightarrow$(ii): Obvious.

  (ii)$\Rightarrow$(iii): Use Proposition \ref{6.0.1}.

 (iii)$\Rightarrow$(i): First, we show that $|S+S|\geq 2$. Assume, for contrary, that $S+S=\{u\}$ for some $u\in S$. Then for every $a\in S$ it holds that $au=a(u+u)=au+au=u$ and, similarly, $ua=u$. Hence $u$ is a bi-absorbing element, a contradiction.
 
 Therefore, $|S+S|\geq 2$. Now, since the semiring $S$ is congruence-simple, $S$ is also bi-ideal-simple. As the set $S+S$ is a non-trivial bi-ideal of $S$, it follows that $S+S=S$. By Proposition \ref{2.6}, the semiring $S$ has a multiplicatively absorbing element $w\in S$.

 We show that $w$ is additively neutral. As $w$ is multiplicatively absorbing, we have $w=w(w+w)=w^2+w^2=w+w$. Clearly, $S+w$ is a bi-ideal of $S$ and $w=w+w\in S+w$. Since $S$ is congruence-simple, $S$ is also bi-ideal-simple. Hence we have either $S+w=S$ or $S+w=\{w\}$. As the latter case means that $w$ is bi-absorbing, it follows, by Proposition \ref{6.0.1}, that $S+w=S$ holds. Thus for every $a\in S$ there is $b\in S$ such that $a=b+w$. It follows that $a+w=b+w+w=b+w=a$ for every $a\in S$ and therefore $w=0_S$ is a zero. By Proposition \ref{6.0.1}, $S$ is either additively cancellative (and thus a ring, by Theorem \ref{unity}) or additively idempotent.

 Now let $\varrho$ be a congruence of $T=\mathbf{M}_n(S)$. Define a relation $\sigma$ on $S$ by $(a,b)\in\sigma$ iff $(E_{ij}(a),E_{ij}(b))\in\varrho$ for all pairs $(i,j)$. It easily follows,  by Lemma \ref{6.2}(i) and (ii), that $\sigma$ is a semiring congruence on $S$.

Assume now that $\varrho\neq id_T$, i.e., there are $A,B\in T$ such that $(A,B)\in\varrho$ and  $a=a_{k\ell}\neq b_{k\ell}=b$ for some pair $(k,\ell)$. By Lemma \ref{6.0}, there are $c,d\in S$ such that $cad\neq cbd$. Now, by Lemma \ref{6.2} (iii), we have that
$$\big(E_{ij}(cad),E_{ij}(cbd)\big)=\big(E_{ik}(c)\cdot A\cdot E_{\ell j}(d),\ E_{ik}(c)\cdot B\cdot E_{\ell j}(d)\big)\in\varrho$$
for all pairs $(i,j)$. Hence $(cad,cbd)\in\sigma$ and therefore $\sigma\neq id_{S}$. Since $S$ is congruence-simple, we obtain that $\sigma=S\times S$. Therefore for all $C=(c_{ij})$ and $D=(d_{ij})$ in $T$ it follows that
$(C,D)=\left(\sum_{i,j}E_{ij}(c_{ij}),\sum_{i,j}E_{ij}(d_{ij})\right)\in\varrho$
and we obtain that  $\varrho=T\times T$. We conclude that $T$ is congruence-simple.
\end{proof}

In the literature and applications, semirings with a zero (and, possibly, a unity) appear often. 
In such a case the characterization of congruence-simple matrix semirings is fully answered by Theorem \ref{6.4}. On the other hand, also semirings without such special elements (e.g., the unbounded distributive lattices or the subsemirings of positive reals) are important on themselves. 

\begin{remark}
 A congruence-simple semiring $S$ with a zero element and such that $|SS|\geq 2$ does not need to possess a unity. And, henceforth, Theorem \ref{6.4} is a proper generalization of \cite[Proposition 5.5]{KNZ-2}, where both the zero and the unity are assumed.
 
 (i) For the additively cancellative case (the ring case, in fact) there exist  non-commutative simple rings without zero-divisors and without a unity (see, e.g., \cite{cohn1,cohn2}).
 
 (ii) For the additively idempotent case consider the following. Let $L$ be a semilattice with the least element $\omega\in L$, the greatest element $\Omega\in L$ and let $|L|\geq 3$. Let $\mathrm{End}_{\omega}(L)$ be the semiring of all endomorphisms of $L$ that preserve $\omega$. Now the semiring $S=\set{\varphi\in \mathrm{End}_{\omega}(L)}{\varphi(L)\ \emph{is a finite set}}$ is a subsemiring of  $\mathrm{End}_{\omega}(L)$ with a zero element. By \cite[Corollary 2.7]{simple-zero}, the semiring $S$ is congruence-simple and it is easy to show that $S$ contains zero-divisors.
 
 If $L$ is infinite, then $S$ does not contain a unity and if $L$ is finite then $S=\mathrm{End}_{\omega}(L)$ has a unity, but there are non-zero elements without (multiplicative) inverses. 
 
\end{remark}

 By Proposition \ref{6.0.1},  the case of congruence-simple matrix semirings  splits into two disjoint and essentially different classes -- one of them being the  additively cancellative semirings (i.e., embeddable into rings) and the other one is the class of additively  idempotent semirings (appearing, e.g., in logics, partially ordered structures or tropical algebra).

The rest of the paper is devoted to the first class mentioned  -- the additively cancellative case. The latter class will be studied  in a follow-up article.

\section{Bi-ideal-free matrix semirings}

As we already observed, every congruence-simple semiring is bi-ideal-simple, and in fact, properties of congru\-ence-simple semirings can be often derived using their bi-ideal-simplicity. In particular, this property plays a key role in the study of congruence-simple semirings that are additively cancellative (see Section \ref{add_cancell}).   In the present section we study the interplay between the semiring $S$ and its matrix semiring $\mathbf{M}_n(S)$ from the viewpoint of their bi-ideal-simplicity (and bi-ideal-freeness).

Through this section let $S$ be a semiring.

\begin{lemma}\label{5.1}
 Let $I$ be an bi-ideal of $\mathbf{M}_n(S)$. Then:
  \begin{enumerate}
  \item[(i)] $\overline{a}\in I$ for at least one $a\in S$.
  \item[(ii)] If $a\in S$ is such that $\overline{a}\in I$ and $a+S=S$, then $I=\mathbf{M}_n(S)$.
 \end{enumerate}
\end{lemma}
\begin{proof}
 (i) Assume $n\geq 2$, take $A=(a_{ij})\in I$ and put $a=\sum_{i,j} a_{ij}$ and $b_{ij}=\sum_{(k,\ell)\neq(i,j)} a_{k\ell}$. Then $A+B=\overline{a}$, and hence $\overline{a}\in I$.

 (ii) If $A=(a_{ij})\in I$, then $a+b_{ij}=a_{ij}$ for some $b_{ij}\in S$. Now, $\overline{a}+B=A$, and hence $A\in I$.
\end{proof}

\begin{lemma}\label{5.3}
 Let  $a\in S$. Then:
   \begin{enumerate}
  \item[(i)] For all $C,D\in \mathbf{M}_n(S)$ there are $c_1,\dots,c_n,d_1,\dots,d_n\in S$ such that $(C\overline{a}D)_{ij}=c_iad_{j}$ for all $i,j$.
  \item[(ii)] Let $c_1,\dots,c_n,d_1,\dots,d_n\in S$ and $F=(c_iad_j)\in \mathbf{M}_n(S)$. Then for every $B\in \mathbf{M}_n(S)$ there are $c'_1,\dots,c'_n,d'_1,\dots,d'_n\in S$ such that $(BF)_{ij}=c'_iad_{j}$ and  $(FB)_{ij}=c_iad'_{j}$ for all $i,j$.
  \item[(iii)] If, moreover $S+S=S$, then for all $c_1,\dots,c_n,d_1,\dots,d_n\in S$ there are $C,D\in \mathbf{M}_n(S)$ such that $(C\overline{a}D)_{ij}=c_iad_{j}$ for all $i,j$.
 \end{enumerate}
\end{lemma}
\begin{proof}
(i) Let $C=(c_{ij}), D=(d_{ij})\in \mathbf{M}_n(S)$. Then $(C\overline{a}D)_{ij}=\sum_{k,\ell} c_{ik}ad_{\ell j}=(\sum_{k} c_{ik})a(\sum_{\ell}d_{\ell j})=c_i ad_j$ for $c_i=\sum_{k} c_{ik}$ and $d_j=\sum_{\ell}d_{\ell j}$.

(ii) and (iii). Let $c_1,\dots,c_n,d_1,\dots,d_n\in S$ and $F=(c_iad_j)\in \mathbf{M}_n(S)$. For $B=(b_{ij})\in \mathbf{M}_n(S)$ set $c'_i=\sum_{k} b_{ik}c_k$ and  $d'_j=\sum_{\ell}d_{\ell}b_{\ell j}$. Then  $(BF)_{ij}= \sum_{k}b_{ik}c_kad_j=( \sum_{k}b_{ik}c_k)ad_j=c'_iad_{j}$ and $(FB)_{ij}=\sum_{\ell}c_iad_\ell b_{\ell j}=c_ia(\sum_{\ell}d_\ell b_{\ell j}) =c_iad'_{j}$. 

Assume, moreover, that $S+S=S$. Then for every $c_i\in S$ there are $c_{ij}\in S$ such that $c_{i}=\sum_{j} c_{ij}$. Similarly, for every $d_j\in S$ there are $d_{ij}\in S$ such that $d_{j}=\sum_{i} d_{ij}$. Now, for $C=(c_{ij})$ and $D=(d_{ij})$ it holds that  $(C\overline{a}D)_{ij}=\sum_{k,\ell} c_{ik}ad_{\ell j}=(\sum_{k} c_{ik})a(\sum_{\ell}d_{\ell j})=c_i ad_j$.
\end{proof}

\bigskip

Let us now introduce and study the properties of certain prominent bi-ideals of matrix semirings.
For that, let $T=\mathbf{M}_n(S)$ and $a\in S$. 
 We set $Q_a=T\cdot \overline{a}\cdot T+T$ and
 $$P_a=\set{(a_{ij})\in T}{\exists\ c_{1},\dots,c_{n},d_{1},\dots,d_{n}\in S\ \ \forall i,j \ \  c_ia d_j\leq_S a_{ij} }$$
  By  $R_a$ we denote the bi-ideal of $T$ generated by $\overline{a}\in T$.

 \begin{lemma}\label{bi-ideals}
 Let $a\in S$. Then $Q_a\sub P_a$ and $Q_a\sub R_a$ and the sets $Q_a$ and $P_a$ are bi-ideals  of $T=\mathbf{M}_n(S)$.
 If moreover $S+S=S$, then $Q_a=P_a$.
 \end{lemma}
 \begin{proof}
 By Lemma \ref{5.3}(i), $Q_a\sub P_a$. Clearly, $P_a+ T\sub P_a$ and, by Lemma \ref{5.3}(ii), $TP_a\sub P_a$ and $P_aT\sub P_a$. Hence $P_a$ is a bi-ideal of $T$.
 
 Assume that $S+S=S$ and let $A=(a_{ij})\in P_a$. Then there are $c_{1},\dots,c_{n},$ $d_{1},\dots,d_{n}\in S$ such that $c_ia d_j\leq_S a_{ij}$ for all $i,j$. As $S+S=S$, there are $e_1,\dots,e_n,f_1,\dots,f_n\in S$ such that $c_i=e_i+f_i$ for every $i$. Hence $a_{ij}\geq_Sc_ia d_j= e_ia d_j+f_ia d_j$ and, therefore, there is $M=(m_{ij})\in T$ such that $a_{ij}= e_ia d_j+m_{ij}$ for all $i,j$. Now, by Lemma \ref{5.3}(iii), there are $E,D\in T$ such that $A=E\overline{a}D+M$. It follows that $A\in Q_a$ and we obtain that $Q_a=P_a$.
 \end{proof}

\begin{lemma}\label{5.5}
 Let $a\in S$. Then $|Q_a|=1$ if and only if $S$ has a bi-absorbing element $o_S$ and $S_{n}a S_{n}+S=\{o_S\}$ (see Proposition $\ref{2.6}$).
\end{lemma}
\begin{proof}
Assume first that $Q_a=\{W\}$ for some $W=(w_{ij})\in T =\mathbf{M}_n(S)$. Since $Q_a$ is a bi-ideal in $T$ of cardinality one, the semiring $T$ has a bi-absorbing element. By Proposition \ref{2.5}(ii), the semiring $S$ has also a bi-absorbing element $o_S$ and $w_{ij}=o_S$ for all $i,j$.  Now, $A\overline{a}B+C=W$ for all $A=(a_{ij})$, $B=(b_{ij})$, $C=(c_{ij})$ in $T$. Hence we get that $o_S=w_{ij}=(\sum_{k}a_{ik})a(\sum_{\ell}b_{\ell j})+c_{ij}$ for all $i,j$.  Consequently,  $S_{n}a S_{n}+S=\{o_S\}$.

The converse implication is easy.
\end{proof}

We can now turn our attention to the relationship between the (non-)existence of proper bi-ideals in $S$ and in $\mathbf{M}_n(S)$.

\begin{proposition}\label{5.9}
 Assume that $n\geq 2$ and  $\mathbf{M}_n(S)$ is bi-ideal-simple. Then the semiring $S$ is bi-ideal-free.
\end{proposition}
\begin{proof}
 Let $I$ be a bi-ideal of $S$. The set $J=\set{A=(a_{ij})\in T}{a_{11}\in I}$ is a bi-ideal of $T=\mathbf{M}_n(S)$. If $J=T$, then $I=S$. If $|J|=1$, then $|I|=1$, $I=\{o_S\}$, where $o_S$ is bi-absorbing and, since $n\geq 2$, we have $|S|=1$ and $S=\{o_S\}$. But this is a contradiction, since $T$ is non-trivial.
\end{proof}

\begin{proposition}\label{5.11}
 Assume that $S$ is bi-ideal-free. Then:
 \begin{enumerate}
  \item[(i)]  If $I$ is a bi-ideal of $\mathbf{M}_n(S)$, then $\overline{S}=\set{\overline{b}}{b\in S}\sub I$.
  \item[(ii)] There is the least bi-ideal $J$ of $\mathbf{M}_n(S)$ and $J=Q_a=R_a=P_a$ for every $a\in S$.
 \end{enumerate}
\end{proposition}
\begin{proof}
 (i)
 Let $I$ be a bi-ideal of $\mathbf{M}_n(S)$. By Lemma \ref{5.1}, $\overline{a}\in I$ for some $a\in S$. Set $a'=n^2a$. The set $Sa'S+S$ is a bi-ideal of $S$, and hence $Sa'S+S=S$. If $b\in S$ is any element, then $b=ca'd+e$ for some $c,d,e\in S$. We have $\overline{b}=\overline{ca'd}+\overline{e}=\overline{n^2cad}+\overline{e}=n^2\overline{cad}+\overline{e}=\overline{c}\cdot\overline{a}\cdot\overline{d}+\overline{e}\in I$. We have proven that $\overline{S}\sub I$.

 (ii) If $I$ is a bi-ideal of $T$, then $\overline{S}\sub I$ by (i). Thus $R_a\sub I$ and $Q_a\sub I$ for every $a\in S$. As both $R_a$ and $Q_a$ are bi-ideals, it follows that $R_a=Q_a$ for every $a\in S$ and this set is the least bi-ideal of $T$. Finally, as $S$ is bi-ideal-free, the bi-ideal $S+S$ is equal to $S$. Hence, by Lemma \ref{5.5}, $Q_a=P_a$ for every $a\in S$. 
\end{proof}

\begin{proposition}\label{5.10}
 Let $n\geq 2$ and $|S|\geq 2$. Then the following conditions are equivalent:
 \begin{enumerate}
  \item[(i)] The matrix semiring $\mathbf{M}_n(S)$ is bi-ideal-simple.
  \item[(ii)] The matrix semiring $\mathbf{M}_n(S)$ is bi-ideal-free.
  \item[(iii)]  $Q_a=\mathbf{M}_n(S)$ for every $a\in S$.
  \item[(iv)] $\forall\ a\in S\ \forall (a_{ij})\in \mathbf{M}_n(S)\  \exists\ c_{1},\dots,c_{n},d_{1},\dots,d_{n}\in S\ \ \forall i,j \ \  c_ia b_j\leq_S a_{ij}$.
 \end{enumerate}
 If these equivalent conditions are satisfied, then $S$ is bi-ideal-free.
\end{proposition}
\begin{proof}
 (i)$\Rightarrow$(iii): Let $a\in S$. If $|Q_a|=1$, then $S$ has a bi-absorbing element by Lemma \ref{5.5}, a contradiction with Proposition \ref{5.9}. Now $Q_a$ is a non-trivial bi-ideal of $T$ and $T$ is bi-ideal-simple. Hence  $Q_a=T$ and  $S$ is bi-ideal-free, by Proposition \ref{5.9}.

 (iii)$\Rightarrow$(ii): Let $I$ be a bi-ideal of $T$. By Lemma \ref{5.1}, $\overline{a}\in I$ for at least one $a\in S$. Since $I$ is a bi-ideal, we get $Q_a\sub I$. But then $I=T$.

 (ii)$\Rightarrow$(i): This is obvious.
 
 (iii)$\Rightarrow$(iv): Follows immediately from Lemma \ref{bi-ideals}.
 
 (iv)$\Rightarrow$(iii): Let $b\in S$  and consider a matrix $A=(a_{ij})\in T$ such that $a_{11}=b$. By the assumption of (iv), there are $c_1,d_1\in S$ such that $c_1(b+b)d_1\leq_S b$. Hence $b\in S+S$ and it follows that $S+S=S$. As the assumption (iv) means that $P_a=T$ for every $a\in S$, we obtain, by Lemma \ref{bi-ideals}, that $Q_a=P_a=T$ for every $a\in S$.
\end{proof}

\begin{corollary}\label{corollary_bi-ideal}
 Let $n,k$ be positive integers such that $n\geq k\geq 2$. If $\mathbf{M}_n(S)$ is bi-ideal-simple, then $\mathbf{M}_k(S)$ is bi-ideal-simple.
\end{corollary}
\begin{proof}
 Follows immediately from Proposition \ref{5.10}(i) and (iv).
\end{proof}

\bigskip

As one can notice, all the conditions (i)--(iv) in Proposition \ref{5.10} depend on $n$. But is it natural to  expect, in view of Theorem \ref{6.4}, that there will be a common equivalent condition for all $n$. In particular, an immediate question suggests itself: Let $n\geq 3$ and $\mathbf{M}_2(S)$ be bi-ideal-simple. Is then $\mathbf{M}_n(S)$  bi-ideal-simple?

In view of Proposition \ref{5.9} and the Propositions \ref{commutative_2} and \ref{5.13.0}, that immediately follow,  it is likely that the condition equivalent to (i)--(iv) in Proposition \ref{5.10} will be of the form
\begin{enumerate}
\item[(v)] $S$ is bi-ideal-free and for all $a,b\in S$ there is $c\in S$ such that $c\leq_S a$ and $c\leq_S b$.
\end{enumerate}

\bigskip

The following propositions are essential for proving our main results  in Section \ref{add_cancell}.

 \begin{proposition}\label{commutative_2}
 Let $S$ be commutative and $\mathbf{M}_n(S)$ be bi-ideal-free for some $n\geq 2$. Then for all $a,b\in S$ there is $c\in S$ such that $c\leq_S a$ and $c\leq_S b$. 
 \end{proposition}
  \begin{proof}
Let $a,b\in S$. By Lemma \ref{bi-ideals}, the set $P_{a+b}$ is a bi-ideal in $T=\mathbf{M}_n(S)$. As this semiring is bi-ideal-free, we have that $P_{a+b}=T$. Consider now a matrix $A=(a_{ij})\in T=P_{a+b}$ such that $a_{11}=a_{22}=a$ and $a_{12}=a_{21}=b$. By the definition of $P_{a+b}$, there are $c_{1},c_{2},d_{1},d_{2}\in S$ such that $c_{1}(a+b)d_1\leq_S a$, $c_{2}(a+b)d_2\leq_S a$, $c_{1}(a+b)d_2\leq_S b$ and  $c_{2}(a+b)d_1\leq_S b$. Now, set $c=c_1c_2(a+b)d_1d_2$. By the commutativity of $S$ it follows that $c=c_1\cdot c_2(a+b)d_2\cdot d_1\leq_S c_1\cdot a\cdot d_1\leq_S c_1(a+b)d_1\leq_S a$ and
$c=c_1\cdot c_2(a+b)d_1\cdot d_2\leq_S c_1\cdot b\cdot d_2\leq_S c_1(a+b)d_2\leq_S b$.
Hence  $c\leq_S a$ and $c\leq_S b$.
 \end{proof}

\begin{proposition}\label{5.13.0}
Assume that for all $a,b\in S$ there is $c\in S$ such that   $c\leq_S a$ and $c\leq_S b$.
 If $S$ is bi-ideal-free, then $\mathbf{M}_n(S)$ is bi-ideal-free.
\end{proposition}
\begin{proof}
First, the set $S+S$ is a bi-ideal of $S$. Since $S$ is bi-ideal-free, it follows that $S+S=S$.  By our assumption,  for all $a,b\in S$ there is  $c\in S$ such that   $c\leq_S a$ and $c\leq_S b$. As $c\in S=S+S$ there are $d,e\in S$ such that $c=d+e$  and therefore $a,b\in d+S$. By induction, we easily obtain that for every $A\in T=\mathbf{M}_n(S)$  there  is $d\in S$ such that $A\in\overline{d}+T$.

 Now, let $I$ be a bi-ideal of $T=\mathbf{M}_n(S)$. By Proposition \ref{5.11}, $\overline{S}\sub I$. By the previous part of the proof, for every $A\in T$ there is $d\in S$ such that $A\in \overline{d}+T$. It follows that $A\in \overline{d}+T\sub I$ and therefore $I=T$. Hence $T$ is bi-ideal simple.
\end{proof}

\begin{proposition}\label{divisible}
 Let $S$ be additively archimedean and additively uniquely $m$-divisible for some $m\geq 2$. Then for all $a,b\in S$ there is $c\in S$ such that  $c\leq_S a$ and $c\leq_S b$. 
\end{proposition}
\begin{proof}
 Assume the conditions on $S$. Let $a,b\in S$. Pick an element $e\in S$. By the achimedean property, there are positive integers $k_1$ and $k_2$ such that $e\leq_{S} k_1 a$ and $e\leq_S k_2 b$. Further, there is $\ell\geq 1$ such that $m^{\ell}\geq 1+\max\{k_1,k_2\}$. Hence  $e\leq_{S} k_1 a +a\leq m^{\ell}a$ and $e\leq_S k_2 b+b\leq_S m^{\ell} b$ and there are $f_1,f_2\in S$ such that $e+f_1=m^{\ell}a$ and $e+f_2=m^{\ell}b$. 
 
 Since $S$ is additively uniquely $m$-divisible, the transformation $\psi:x\mapsto m^{\ell}x$ is an automorphism of $S(+)$.  Therefore there are  $c,d_1,d_2\in S$ such that $m^{\ell}c=e$, $m^{\ell}d_1=f_1$ and $m^{\ell}d_2=f_2$. Now, the equalities $m^{\ell}(c+d_1)=m^{\ell}a$ and $m^{\ell}(c+d_2)=m^{\ell}b$ imply, by the the divisibility uniqueness, that $c+d_1=a$ and $c+d_2=b$. We conclude that  $c\leq_S a$ and $c\leq_S b$.
\end{proof}

\section{Congruence-simple matrix semirings -- The additively cancellative case}\label{add_cancell}

Recall that for an additively cancellative  semiring $S$ there exists a unique ring $R$ (up to an isomorphism) that contains $S$ and is minimal with this property. Such a ring is called the \emph{difference ring} and is denoted by  $R=S-S$. Note that then the difference ring of $\mathbf{M}_{n}(S)$ is $\mathbf{M}_{n}(R)$.

An additively cancellative semiring $S$ is called \emph{conical} if the difference ring $R=S-S$ is simple. 

Using this terminology, one can then formulate the following key classification result.

\begin{theorem}\cite[Theorem 5.1]{simple}\label{bashir}
Let $S$ be an additively cancellative semiring. Then
$S$ is congruence-simple if and only if the following three conditions are satisfied:
\begin{enumerate}
\item[(1)] $S$ is conical,
\item[(2)] $S$ is bi-ideal-simple,
\item[(3)] $S$ is additively archimedean.
\end{enumerate}
\end{theorem}

\begin{proposition}\label{ordering}
  Let $S$ be a congruence-simple additively cancellative semiring that is not a ring.  Let $R=S-S$ be the difference-ring of $S$. Then the relation $\leq$ on $R$ defined by $r\leq s$ for $r,s\in R$ if and only if $r=s$ or $s-r\in S$ is a (partial) ordering on $R$. Moreover:
  \begin{enumerate}
   \item[(i)] The restriction of $\leq$ to $S\times S$ is $\leq_S$.
   \item[(ii)] For all $r,s,t,u\in R$, $0<u$, the condition $r\leq s$ implies that $r+t\leq s+t$, $ru\leq su$ and $ur\leq us$. 
   \item[(iii)] For all $r,u\in R$, $0<u$, there is $n\in\N$ such that $r\leq nu$.
  \end{enumerate}
In particular, $(R,\leq)$ is an ordered simple ring.
\end{proposition}
\begin{proof}
 Let $\varrho$ denote the kernel of the standard quasiordering $\leq_S$. That is, $(a,b)\in\varrho$ if and only if either $a=b$ or both $a\in b+S$ and $b\in a+S$ hold. It is immediately clear that $\varrho$ is a congruence of the semiring $S$. If $\varrho=S\times S$, then $S$ is a ring, a contradiction. Hence $\varrho\neq S\times S$ and, as $S$ is congruence-simple, it follows that $\varrho=id_S$ and the quasiordering $\leq_S$ is an ordering of the semiring $S$. Hence the ordering $\leq$ on $R$, induced by the ordering $\leq_S$, is of the form $r\leq s$ for $r,s\in R$ if and only if $r=s$ or $s-r\in S$. Now the conditions (i) and (ii) follow easily from the properties of $\leq_S$. Finally, the condition (iii) follows from the archimedean property of $S$, by Theorem \ref{bashir}.
\end{proof}

\begin{remark}\label{ordering_2}
 Let $S$ be a congruence-simple additively cancellative semiring that is not a ring. 
 Let $R=S-S$ be the difference ring of $S$ with the ordering $\leq$ from Proposition \ref{ordering}. Note that the ordering $\leq_S$ on $S$ is always \emph{upwards directed} (as we have $a,b\leq_S a+b$). 
 
 Now if we assume that for all $a,b\in S$ there is $c\in S$ with $c\leq_S a$ and $c\leq_S b$ (i.e., the ordering $\leq_S$ on $S$ is \emph{downwards directed}), then the ordering $\leq$ is both upwards and downwards directed on $R$:

 Indeed, take $r,s\in R$, $r=a-b$, $s=c-d$, where $a,b,c,d\in S$. Then $r,s\leq a+c$ $(\in S)$. Furthermore, $e\leq_S a,b$ for some $e\in S$ and henceforth $t\leq r,s$, where $t=e-(b+d)$$(\in R)$.  
\end{remark}

Next, let us show the extent to which the properties from Theorem \ref{bashir} hold for matrix semirings.

\begin{proposition}\label{3.1}
Let $S$ be a semiring. Then the matrix semiring $\mathbf{M}_n(S)$ is additively archimedean if and only if the semiring $S$ is such.
\end{proposition}
\begin{proof}
 If $S$ is additively archimedean and $A,B\in T=\mathbf{M}_n(S)$, then there exist $c_{ij}\in S$ and positive integers $m_{ij}$ such that $a_{ij}+c_{ij}=m_{ij}b_{ij}$. Setting $d_{ij}=c_{ij}+(m-m_{ij})b_{ij}$, where $m=\max\set{m_{ij}}{i,j}+1$, we get $A+D=mB$.

 Conversely, if $T$ is additively archimedean and $a,b\in S$, then $\overline{a}+C=k\overline{b}$ for a positive integer $k$ and a matrix $C=(c_{ij})\in T$. Now, $a+c_{ij}=kb$ and $S$ is therefore additively archimedean.
\end{proof}

\begin{proposition}\label{4.1}
 Let $S$ be an additively cancellative semiring.  Then the following conditions are equivalent:
 \begin{enumerate}
   \item[(i)] $\mathbf{M}_n(S)$ is conical for every $n\geq 2$.
  \item[(ii)] $\mathbf{M}_2(S)$ is conical.
  \item[(iii)] $S$ is conical and $|SS|\geq 2$.

 \end{enumerate}

\end{proposition}
\begin{proof}
(i)$\Rightarrow$(ii): Obvious.

 (ii)$\Rightarrow$(iii): Assume the condition (ii). Then the difference ring $\mathbf{M}_2(R)=\mathbf{M}_2(S)-\mathbf{M}_2(S)$  is simple. By Theorem \ref{6.4}, the difference ring $R=S-S$ is simple and $RR\neq\{0_R\}$. Hence  $S$ is conical and $|SS|\geq 2$.

 (iii)$\Rightarrow$(i): Assume the condition (iii). Then the difference ring $R=S-S$ is simple. If $RR=\{0_R\}$ then $R$ is finite of prime order, and therefore $S=R$ and $|SS|=1$, a contradiction. Thus $|RR|\geq 2$. Hence, by Theorem \ref{6.4}, the ring $\mathbf{M}_n(R)=\mathbf{M}_n(S)-\mathbf{M}_n(S)$ is simple. This means that the semiring $\mathbf{M}_n(S)$ is conical.
\end{proof}

Now we are ready to prove our main results on congruence-simple additively cancellative matrix semirings.

\begin{theorem}\label{theorem_cancellative_2}
 Let $S$ be an additively cancellative semiring. Assume that at least one of the following two conditions holds:
 \begin{enumerate}
  \item[(a)] For all $a,b\in S$ there is $c\in S$ such that  $c\leq_S a$ and $c\leq_S b$ (i.e., the ordering $\leq_S$ is downwards directed, cf. Remark $\ref{ordering_2}$).
  \item[(b)] $S$ is additively $m$-divisible for some $m\geq 2$.
 \end{enumerate}
 Then the following are equivalent:
 \begin{enumerate}
  \item[(i)] $\mathbf{M}_n(S)$ is congruence-simple for every $n\geq 1$.
  \item[(ii)] $\mathbf{M}_n(S)$ is congruence-simple for at least one $n\geq 2$.
  \item[(iii)] $S$ is congruence-simple and $|SS|\geq 2$.
 \end{enumerate}
\end{theorem}
\begin{proof}
  (i)$\Rightarrow$(ii): Obvious.

  (ii)$\Rightarrow$(iii): Follows from Proposition \ref{6.0.1}. 
  
  (iii)$\Rightarrow$(i): Assume the condition (iii). For $k\geq 2$ the relation $\varrho_k$ on $S\times S$, defined as $(x,y)\in\varrho_k$ if and only $kx=ky$, is a semiring congruence on $S$. 
  
  If $\varrho_k=S\times S$ then the element $w=kx$ is multiplicatively absorbing (indeed, $yw=y(kx)=k(yx)=w$ and, similarly, $wy=w$ for all $x,y\in S$). As $S$ is embeddable into a ring and $|S|\geq 2$, the semiring $S$ has no bi-absorbing element. Thus, the semiring $\mathbf{M}_n(S)$ is congruence-simple for every $n\geq 1$, by Theorem \ref{6.4}.
  
  Assume therefore that $\varrho_k\neq S\times S$ for every $k\geq 2$. Since $S$ is congruence-simple, it follows that $\varrho_k=id_S$ for every $k\geq 2$. Therefore, if $S$ is additively $m$-divisible for some $m\geq2$, then $S$ is additively uniquely $m$-divisible.
  
   Hence, in view of Proposition \ref{divisible} and assumptions (a) or (b), we obtain that for every $a,b\in S$ there is $c\in S$ such that $c\leq_S a$ and $c\leq_S b$.
   
  To prove now that for given $n\geq 1$ the semiring  $\mathbf{M}_n(S)$ is congruence-simple,   we need to verify the conditions (1), (2) and (3) in Theorem \ref{bashir}.
 
 First, since $S$ is congruence-simple, $S$ is conical by Theorem \ref{bashir}(1). Hence, by Proposition \ref{4.1}, $\mathbf{M}_n(S)$ is conical too.
 
 Further, since $S$ is embeddable into a ring, $S$ can not contain a bi-absorbing element and therefore every bi-ideal of $S$ has at least two elements. Hence, as the semiring $S$ is bi-ideal-simple, by Theorem \ref{bashir}(2), it follows that $S$ is bi-ideal-free. Therefore, by Proposition \ref{5.13.0},  $\mathbf{M}_n(S)$ is bi-ideal-free.
 
 Finally, by Theorem \ref{bashir}(3), $S$ is additively archimedean. Thus, by Proposition \ref{3.1}, $\mathbf{M}_n(S)$ is additively archimedean, too.
 
We conclude that $\mathbf{M}_n(S)$ is congruence-simple, by Theorem \ref{bashir}.
\end{proof}

In the case that the semiring $S$ is commutative  we can make the characterization from Theorem \ref{theorem_cancellative_2} more precise.

\begin{theorem}\label{theorem_cancellative}
 Let $S$ be an additively cancellative semiring. Assume that at least one of the following two conditions holds:
 \begin{enumerate}
  \item[(a)] $S$ is commutative.
  \item[(b)] $S$ has a unity or an additively neutral element.
 \end{enumerate}
 Then the following are equivalent:
 \begin{enumerate}
  \item[(i)] $\mathbf{M}_n(S)$ is congruence-simple for every $n\geq 1$.
  \item[(ii)] $\mathbf{M}_n(S)$ is congruence-simple for at least one $n\geq 2$.
  \item[(iii)] $S$ is congruence-simple, $|SS|\geq 2$ and for all $a,b\in S$ there is $c\in S$ such that  $c\leq_S a$ and $c\leq_S b$.
 \end{enumerate}
\end{theorem}
\begin{proof}
  (i)$\Rightarrow$(ii): Obvious.

  (ii)$\Rightarrow$(iii): By Proposition \ref{6.0.1}, $S$ is congruence-simple and $|SS|\geq 2$. Assume that either the condition (a) or (b) holds. Let $a,b\in S$. 
  We show that there is $c\in S$ such that $c\leq_S a$ and $c\leq_S b$.
  
  If $S$ has an additively neutral element $\omega\in S$, then $\omega\leq_S a$ and $\omega\leq_S b$. If  $S$ is without an additively neutral element, then $S$ is not a ring, and, by Theorem \ref{unity},  the semiring $S$ is commutative. Now, the semiring $\mathbf{M}_n(S)$ is bi-ideal-simple and, by Proposition \ref{5.10}, also bi-ideal-free. Therefore, by Proposition \ref{commutative_2}, there is $c\in S$ such that $c\leq_S a$ and $c\leq_S b$.

 (iii)$\Rightarrow$(i): Follows from Theorem \ref{theorem_cancellative_2}.
\end{proof}

In view of Theorems \ref{theorem_cancellative_2} and \ref{theorem_cancellative} we propose the following conjecture.

\bigskip
\textbf{Conjecture:} Let $S$ be an additively cancellative semiring. Then the conditions (i), (ii) and (iii) from Theorem \ref{theorem_cancellative} are equivalent.
\bigskip

Let us note that,  by \cite[Theorem 10.1]{simple_comm}, any additively cancellative congruence-simple semiring, that is \emph{commutative} but \emph{not} a ring, is a subsemiring of the positive reals $\R^+$. As it was shown in \cite{cong-simple}, even the semiring $\Q^+$  of positive rationals contains  uncountably many such subsemirings. For the list found in \cite{cong-simple} it is still not clear whether it is complete or not and it seems that congruence-simple subsemirings of $\R^+$ are quite far from being understood. 

\bigskip

The following example shows that a lot of  congruence-simple subsemirings of $\Q^+$ fulfills the condition (iii) in Theorem \ref{theorem_cancellative} and, henceforth,  their corresponding matrix semiring provide a wide class of examples of non-commutative congruence-simple semirings. 

 The semirings in Example \ref{ex_1}  fulfill assumptions (a) in both Theorems \ref{theorem_cancellative_2} and \ref{theorem_cancellative} but do not fulfill any of the conditions (b) in them. Also Theorem \ref{6.4} can not be applied for these semirings.
 
\begin{example}\label{ex_1}
 Consider the following subsemiring of the field of rational numbers \cite{cong-simple} $$S=\set{x\in\Q^+}{a^{v_p(x)} < x}$$ where $0<a<1$ is a real number, $p$ is a prime number and $v_p:\Q^+\to\Z$ the $p$-adic valuation.

 This semiring $S$ has the following properties:
 \begin{enumerate}
 \item[(i)] $S$ is congruence-simple (see \cite[Theorem 2.16]{cong-simple}).
 \item[(ii)] $S$ is without both a zero and a unity.
 \item[(iii)] For all $x,y\in S$ there is $z_0\in S$ such that $z_0\leq_S x$ and $z_0\leq_S y$.
 \item[(iv)] $\mathbf{M}_n(S)$ is congruence-simple and without both a zero and a unity (use Theorem \ref{theorem_cancellative} and Propositions \ref{2.5} and \ref{2.4}).
 \item[(v)] $S$ is not additively $m$-divisible for any $m\geq 2$.
 \end{enumerate}

 We prove the property (iii).  For $x,y\in S$ there is $0<\varepsilon\in\R$ such that $a^{v_p(x)} < x-\varepsilon$ and $a^{v_p(y)} < y-\varepsilon$. Now, there is $k\in\N$ such that $k> v_p(x)$, $k>v_p(y)$ and $a^k<\varepsilon$. By  \cite[Lemma 1.9]{cong-simple}, the set $\set{z\in\Q^+}{v_{p}(z)=k}$ is dense in $\Q^+$ (in the usual topology). Hence there is $z_0\in\Q^+$ such that $a^k<z_0<\varepsilon$ and $v_p(z_0)=k$ and, therefore, we have that $z_0\in S$.

  Further, obviously $x-z_0\in\Q^+$. Since $v_p(z_0)> v_p(x)$ we obtain by the basic properties of $p$-adic valuation that $v_p(x-z_0)=\min\{v_p(x),v_p(z_0)\}=v_p(x)$. Thus $a^{v_p(x-z_0)}=a^{v_p(x)}<x-\varepsilon<x-z_0$ and it follows that $x-z_0\in S$, i.e., $z_0\leq_S x$. By a similar argument we obtain that $z_0\leq_S y$.

 Finally, we prove the property (v). Let $m\geq 2$. It is enough to find $x\in S$ such that $\frac{x}{m}\notin S$.

 We have $v_p(m)\geq 0$ and therefore $1<\frac{m}{a^{v_p(m)}}$. Now, choose an arbitrary $k\in\Z$. It follows that $a^{k}<a^{k}\frac{m}{a^{v_p(m)}}=ma^{k-v_p(m)}$ and by the density argument (see above) there is $x\in\Q^+$ such that $a^k <x<ma^{k-v_p(m)}$ and $v_p(x)=k$. Hence $x\in S$ and, as it holds that $\frac{x}{m}<a^{v_p(x)-v_p(m)}=a^{v_p(x/m)}$, we obtain that $\frac{x}{m}\notin S$.
\end{example}

\begin{example}
 Let $S'=\mathbf{M}_k(S)$, where $k\geq 2$ and $S$ is the semiring from Example \ref{ex_1}. Then one can easily check that, by the properties of Example \ref{ex_1}, $S'$ is a \emph{non-commutative} congruence-simple semiring that fulfills assumption (a) in Theorem \ref{theorem_cancellative_2} but do neither fulfill the condition (b) in the same theorem nor any of the conditions (a) or (b) in Theorem \ref{theorem_cancellative}. Also Theorem \ref{6.4} can  not be applied for $S'$. 
\end{example}

\begin{remark}\label{classification}
 Let $S$ be a finite additively cancellative semiring and $n\geq 2$. Then $\mathbf{M}_n(S)$ is congruence-simple if and only if  $S\cong\mathbf{M}_k(\mathbb{F}_q)$ for some finite field $\mathbb{F}_q$ and $k\geq 1$.
 
 Indeed, a finite additively cancellative semiring $S$ has to be a ring (since a finite cancellative semigroup is a group).  If $\mathbf{M}_n(S)$ is congruence-simple, then  $S$ is congruence-simple and $|SS|\geq 2$ by Theorem \ref{6.0.1}. By the basic classification of finite congruence-simple  semirings \cite[Theorem 12]{monico} and the list of all two-element semirings in \cite{simple} it follows that $S\cong\mathbf{M}_k(\mathbb{F}_q)$ for some finite field $\mathbb{F}_q$ and $k\geq 1$.
 The opposite implication is easy. 
\end{remark}

\end{document}